\documentclass[11pt]{article}
\usepackage[T1]{fontenc}
\usepackage[utf8]{inputenc}
\usepackage{lmodern,microtype,booktabs,enumitem,tikz, comment}
\usepackage[a4paper,margin=1in]{geometry}
\usepackage{mathtools,amssymb,amsthm}
\usepackage{xcolor}
\usepackage[colorlinks=true,linkcolor=blue!45!black,citecolor=blue!45!black,urlcolor=blue!45!black]{hyperref}

\hypersetup{pdftitle={
Helly theorems for lattice of flats of a matroid
},pdfauthor={Amir Jafari and Kiumars Kaveh},pdfsubject={Helly theorems for lattice of flats of a matroid},pdfkeywords={matroid, adapted basis, Helly number, lattice of flats, distributive lattice, Bergman fan}}

\newtheorem{theorem}{Theorem}[section]
\newtheorem{lemma}[theorem]{Lemma}
\newtheorem{proposition}[theorem]{Proposition}
\newtheorem{corollary}[theorem]{Corollary}
\newtheorem*{maintheorem}{Main theorem}
\theoremstyle{remark}
\newtheorem{remark}[theorem]{Remark}

\theoremstyle{definition}
\newtheorem{definition}[theorem]{Definition}

\theoremstyle{definition}
\newtheorem{example}[theorem]{Example}

\DeclareMathOperator{\rk}{rk}
\DeclareMathOperator{\cl}{cl}
\newcommand{\cF}{\mathcal F}
\newcommand{\cI}{\mathcal I}
\numberwithin{equation}{section}
\title{Helly theorems for lattice of flats of a matroid}
\author{Amir Jafari\\[2pt]
\small Department of Mathematics\\
\small New Uzbekistan University\\
\small Tashkent, Uzbekistan\\
\small\href{mailto:a.jafari@newuu.uz}{\texttt{a.jafari@newuu.uz}}
\and Kiumars Kaveh\\[2pt]
\small Department of Mathematics\\
\small University of Pittsburgh\\
\small Pittsburgh, PA 15260, USA\\
\small\href{mailto:kaveh@pitt.edu}{\texttt{kaveh@pitt.edu}}}
\begin{document}
\maketitle

\begin{abstract}
A basis of a matroid is adapted to a flat if it contains a basis of the 
flat. We prove that a family of flats in a rank $r$ matroid has a common
adapted basis whenever every subfamily of at most $r+1$ flats does.
The ground set and the family may be infinite; only the rank must be
finite. This result, in particular,
contains Klyachko's vector space Helly theorem for subspaces. As an application, we obtain 
%sharp uniform Helly bounds for flags and apartments in a Bergman fan, and 
a splitting criterion for tropical toric vector bundles on projective spaces. As a secondary application, we derive quadratic bounds for the
Helly number for distributive generation in the lattice of flats.
\end{abstract}


\section{Introduction}
Helly’s classical theorem asserts that a finite family of convex subsets of \(\mathbb R^n\) has nonempty intersection if every subfamily of at most \(n+1\) members has nonempty intersection \cite{Helly1923}. %It is an important example of a general principle which states that a global compatibility condition can be checked on subfamilies of bounded size. 
In this paper, we prove an analogue of the Helly theorem for the lattice of flats of a matroid.%, with the existence of a common adapted basis as the compatibility condition.

Let \(M\) be a matroid of finite rank \(r\), with ground set \(E\). A basis \(B\) of \(M\) is said to be \emph{adapted} to a flat \(F\) if \(B\cap F\) is a basis of the restriction \(M|F\).
A family of flats has a \emph{common adapted basis} if there is a basis of \(M\) that is adapted to every flat in the family. Our main result gives a sharp Helly theorem for the existence of such a basis.

\begin{theorem}[Helly theorem for adapted bases] \label{th-intro-Helly}
Let \(M\) be a matroid of finite rank \(r\), and let \(\mathcal F\) be any family of its flats. If every subfamily of at most \(r+1\) members has a common adapted basis, then \(\mathcal F\) has a common adapted basis. For every \(r\geq2\), the number \(r+1\) is optimal.
\end{theorem}

Neither the ground set nor the family of flats is required to be finite. For each flat \(F\), let \(\mathcal B_F\) denote the set of bases adapted to \(F\). Theorem \ref{th-intro-Helly} says that
\(
\bigcap_{F\in\mathcal F}\mathcal B_F
\)
is nonempty whenever the intersection of every at most \(r+1\) of these sets is nonempty. Thus it is a Helly theorem in the literal intersection sense. We define the Helly number \(h_{\mathrm{cb}}(r)\) to be the smallest number $h$ such that the following holds: for any family of flats, whenever any subfamily of at most $h$ flats has a common adapted basis then the whole family has a common adapted basis (cf. \cite[Section 6]{HuangKaveh2024}). Theorem \ref{th-intro-Helly} says that
\[
h_{\mathrm{cb}}(r)=r+1\qquad(r\geq2).
\]
%Sharpness is witnessed by the \(r+1\) singleton flats of \(U_{r,r+1}\), the cycle matroid of a cycle with \(r+1\) edges: every \(r\) of these flats have a common adapted basis, whereas all \(r+1\) do not.
To show that the $r+1$ bound is sharp, consider the uniform matroid \(M = U_{r,r+1}\) and take the family of $r+1$ singleton flats in $M$. 

Our motivation comes from a remarkable theorem of Klyachko \cite[Theorem 6.1.2]{Klyachko1990}, which gives the same \(r+1\) criterion for a finite collection of subspaces of an \(r\)-dimensional vector space. Here a basis is adapted to a subspace when the subspace is spanned by a subset of the basis. Taking the matroid of all vectors in the ambient space recovers this statement from Theorem \ref{th-intro-Helly}.
%while the matroid formulation also applies to nonrepresentable matroids. 
We remark that Klyachko’s proof uses the modularity of dimension:
\[
\dim U+\dim V=\dim(U\cap V)+\dim(U+V).
\]
But for an arbitrary matroid, the rank function satisfies only the submodular inequality
\[
\operatorname{rk}(F)+\operatorname{rk}(G)
\geq
\operatorname{rk}(F\cap G)+\operatorname{rk}(F\vee G).
\]
%and equality can fail. 
%Extending the argument therefore requires a way to control dependencies without assuming this equality.
We prove Theorem \ref{th-intro-Helly} by induction on rank and restriction to proper flats. %Local adapted bases allow us to construct a finite set \(Q\) whose intersection with every nonempty intersection of the original flats is a basis of that intersection. If \(Q\) were dependent, a circuit in \(Q\), together with inclusion–exclusion, would produce at most \(r+1\) original flats having no common adapted basis. The local hypothesis rules this out, so \(Q\) is independent and extends to the required basis.

%The theorem also applies to flags of flats and yields a common-apartment criterion for Bergman fans: a collection of points in the Bergman fan of a finite loopless rank \(r\) matroid lies in one apartment whenever each subcollection of at most \(r+1\) points does. 

Theorem \ref{th-intro-Helly} has an application to tropical toric vector bundles. A toric vector bundle is a vector bundle on a toric variety together with a linear action of the torus. They have been classified in \cite{Klyachko1990}. Motivated by Klyachko's classification, Khan-Maclagan \cite{KhanMaclagan2024} and Kaveh-Manon \cite{KavehManon2024} introduce a tropical version of toric vector bundles called \emph{tropical toric vector bundles}. 

The Bergman fan of a matroid is a union of subcomplexes, one for each matroid basis, called \emph{apartments} (see \cite{KavehManon2024}). Each apartment can be identified with the Coxeter complex of type A. In \cite{KavehManon2024} a tropical toric vector bundle is described by a map from a fan to the Bergman fan of a matroid that is integral and piecewise linear in the sense that each cone maps to a an apartment in an integral linear way.

A toric vector bundle is said to be \emph{split} if it is equivariantly isomorphic to a sum of toric line bundles. A toric vector bundle is split if the image of the corresponding piecewise linear map lands in one apartment in the extended Tits building (see \cite{Klyachko1990, KavehManon2025}). Motivated by this, a tropical toric vector bundle is said to be \emph{split} when the image of its piecewise linear map lands in one apartment in the Bergman fan. 
%in this case its extension class contains a direct sum of tropical line bundles. We obtain the following consequence.

\begin{corollary}
Every tropical toric vector bundle of rank \(r<n\) on the fan of \(\mathbb P^n\) is split.    
\end{corollary}

%The geometry of the fan of \(\mathbb P^n\) is essential here: any at most \(n\) rays belong to a single cone. When \(r<n\), this supplies the local common-apartment condition for the images of the rays. The Helly theorem, together with the linearity properties on each cone, then gives splitting. The rank bound is optimal for \(n\geq2\).

A second theme of the paper concerns distributivity. For a finite collection of vector subspaces, the existence of a common adapted basis is equivalent to distributivity of the sublattice they generate under intersection and sum \cite[\S 2.2.2]{Klyachko1990}. This equivalence is a main step in Klyachko’s proof. For a matroid, a common adapted basis for a family of flats implies the distributivity of the lattice they generate: the flats spanned by subsets of a fixed basis form a Boolean sublattice of \(L(M)\). The converse, however, is not true in general. For example, take the uniform matroid \(U_{3,4}\) on \(E=\{e_1,e_2,e_3,e_4\}\), the flats
\[
F=\{e_1,e_2\},\qquad G=\{e_3,e_4\}
\]
generate the four-element Boolean lattice \(\{\varnothing,F,G,E\}\), but a basis adapted to both would have to contain all four elements. Proposition \ref{adapted-criterion} identifies the precise additional condition: a finite family of flats has a common adapted basis if and only if the lattice it generates is distributive and rank is a valuation on that lattice.

This distinction motivates a separate Helly problem. For \(r\geq2\), define the \emph{distributive Helly number} \(h_{\mathrm{dist}}(r)\) to be the least positive integer \(h\) such that, for every rank \(r\) matroid and every family \(\mathcal F\) of its flats, the lattice generated by \(\mathcal F\) is distributive whenever every subfamily of at most \(h\) members generates a distributive lattice. Our second main result establishes upper and lower bounds and determines the growth of this number up to an additive constant.

\begin{theorem}[Bounds for distributive generation]
 For every \(r\geq2\),
\[
\begin{aligned}
\max\left\{r+1,\left\lfloor\frac{(r-1)^2}{4}\right\rfloor\right\}
&\leq h_{\mathrm{dist}}(r)\\
&\leq
\max\left\{
r+1,\;2r-3,\;
\left\lfloor\frac{(r-1)^2}{4}\right\rfloor+1
\right\}.
\end{aligned}
\]   
\end{theorem}

\begin{remark}
    When $r>7$, the previous theorem says that
    \[
    \left\lfloor\frac{(r-1)^2}{4}\right\rfloor\le h_{\mathrm{dist}}(r)\le \left\lfloor\frac{(r-1)^2}{4}\right\rfloor+1.\]
    It can be proved that in fact $h_{\mathrm{dist}}(r)$ is equal to the above lower bound for $r>6$, and is equal to $r+1$ for $r\le 6$. However, for the sake of brevity, we have omitted this proof. 
    
\end{remark}
%The same bounds hold when the uniform Helly number is restricted to graphic matroids, or to matroids representable over any prescribed field.

%In particular, for \(r\geq8\) the bounds differ by one:
%\[
%\left\lfloor\frac{(r-1)^2}{4}\right\rfloor
%\leq h_{\mathrm{dist}}(r)
%\leq
%\left\lfloor\frac{(r-1)^2}{4}\right\rfloor+1.
%\]
%Thus
%\[
%h_{\mathrm{dist}}(r)=\frac{(r-1)^2}{4}+O(1),
%\]
%whereas \(h_{\mathrm{cb}}(r)=r+1\). The upper bound combines Theorem 1.1 with a bound on irredundant generators of finite distributive lattices. The lower bound comes from explicit families of flats in graphic matroids. Consequently, the different growth rates persist even for matroids representable over every field.

\medskip
\noindent\textit{Related work.}
Huang and Kaveh \cite[\S6]{HuangKaveh2024} introduced Helly numbers of
buildings, motivated by the splitting problem for toric principal
bundles, and asked for sharp bounds. 
Tsvelikhovskiy \cite{Tsvelikhovskiy2026} has
recently studied Helly numbers for Tits buildings of connected reductive groups and the
splitting of toric principal bundles. In another recent work of Tsvelikhovskiy et al \cite{Kashyap2026}, an analogue of Klyachko's theorem is proved for modules over a DVR. Equivalently, this work computes the Helly number of the Bruhat-Tits building in type A. Theorem~\ref{th-intro-Helly} can also be read in terms of
common basis complexes. Rognes introduced the common basis complex of a
vector space \cite{Rognes1992}, Brück, Piterman, and Welker treat its
analogue for matroids \cite[\S3.2]{BruckPitermanWelker2024}, and
Piterman, Shareshian, and Welker extend common basis complexes to
spherical buildings \cite{PitermanShareshianWelker2025}. These papers
are mainly concerned with the topology of these complexes and of
related posets of decompositions, and Theorem~\ref{th-intro-Helly} is a
statement of a different flavor. Take as vertices the flats of $M$ other
than $0_M$ and $E$, and as faces the finite sets of such flats with a
common adapted basis. Theorem~\ref{th-intro-Helly} says that every
minimal nonface of this complex has at most $r+1$ vertices.

\medskip
\noindent\textit{Organization.}
Section \ref{sec:basics} recalls the necessary matroid and lattice terminology. Section \ref{sec:adapted} proves the Helly theorem for adapted bases and its applications to flags and tropical toric vector bundles. Section \ref{sec:example} examines the relation between distributivity and existence of common adapted bases, and Section \ref{sec:distributive} proves the bounds for distributive generation.

\section*{AI disclosure}
AI tools have been used in major ways for both the write-up and proofs of the results in the paper. All the content has been carefully checked by the authors. In particular, the problems addressed were posed by the authors. The authors take full responsibility for the correctness of the results.

\section{Background on matroids and lattices}\label{sec:basics}

Throughout, we work with matroids of finite rank, but the ground sets are allowed to be infinite.
%Matroids may have arbitrary ground sets, but their ranks are finite.
We write $E$ for the ground set, $\rk$ for the rank function, and $\cl$ for
closure. We use the usual independent set axioms. 
For standard matroid terminology, see \cite{Oxley2011}.

\begin{definition}
For $r\ge2$, the (uniform) \emph{Helly number} $h_{\mathrm{cb}}(r)$
is the least positive integer $h$ with the
following property: for every rank $r$ matroid and every family of its flats, if each subfamily with at most $h$ members has a
common adapted basis, then so does the entire family. (Here $\mathrm{cb}$ stands for \emph{common basis}.)   
\end{definition}

The Helly number is uniform over all matroids; the least number for a fixed matroid can be smaller. The main result of the paper (Theorem \ref{adapted-helly}) shows that the uniform Helly number $h_\mathrm{cb}(r)$ is equal to $r+1$.

The lattice of flats $L(M)$ has operations
\[
 F\wedge G=F\cap G,\qquad F\vee G=\cl(F\cup G).
\]
Its bottom flat is $0_M=\cl(\varnothing)$, the set of loops; its top flat is
$E$. Empty joins mean $0_M$. We retain loops throughout. The restriction
to a flat $F$ is denoted by $M|F$; its flats are exactly the flats of $M$
contained in $F$. A basis of $F$ means a basis of $M|F$, and is \emph{adapted}
to a flat $W\subseteq F$ when its intersection with $W$ is a basis of $W$.

For a nonempty family $\cF$ of flats, $\langle\cF\rangle$ denotes the
sublattice generated by meets and joins. We put
$\langle\varnothing\rangle=\varnothing$ and regard the empty family as
distributive. We do not automatically
adjoin $0_M,E$. Note that adjoining
a bottom or top flat does not affect distributivity. It can affect
height, defined as the number of strict steps in a longest chain.
Every sublattice of $L(M)$ has height at most $\rk(M)$, since strict
inclusion of flats increases rank. A finitely generated
distributive lattice is finite: every term is a join of meets of subsets
of its generators.

\begin{definition}
Similarly, for $r\ge2$, we define the (uniform) \emph{distributive Helly number}  $h_{\mathrm{dist}}(r)$ to be the least positive integer $h$ with the
following property: for every rank $r$ matroid and every
family of its flats, if each subfamily of at most $h$ members is distributive, then so is the entire family.
\end{definition}
 
We note that if $S,T$ are subsets of a basis, then
\begin{equation}\label{coordinate-identities}
 \cl(S)\cap\cl(T)=\cl(S\cap T),\qquad
 \cl(S)\vee\cl(T)=\cl(S\cup T).
\end{equation}
%Indeed, $\rk(\cl(S)\cup\cl(T))=|S\cup T|$, so submodularity bounds
%the rank of the intersection by $|S|+|T|-|S\cup T|=|S\cap T|$.
%The intersection is a flat containing $\cl(S\cap T)$ of that same
%rank, hence equals it. The join identity follows from closure.
A basis $B$ is adapted to $F$ precisely when $F=\cl(B\cap F)$.
Also $\cl(S)\cap B=S$, so the map $S\mapsto\cl(S)$ is injective on
subsets of $B$. Thus the flats to which $B$ is adapted form a Boolean
sublattice of height $r$ with exactly $2^r$ elements. Every family with
a common adapted basis generates a distributive sublattice. In particular, these
facts apply to $M|F$ for every flat $F$. %The identities
%\eqref{coordinate-identities} extend by induction to any finite
%collection of subsets of a basis.

\section{A Helly theorem for common adapted bases}\label{sec:adapted}
In this section we prove the main result of the paper:
\begin{theorem}\label{adapted-helly}
Let $M$ be a matroid of finite rank $r$, and let $\cF$ be any family of
its flats. If every subfamily of at most $r+1$ members has a common adapted
basis, then $\cF$ has a common adapted basis. Moreover, for $r\ge2$
the number $r+1$ cannot be replaced by $r$, already for graphic
matroids; thus $h_{\mathrm{cb}}(r)=r+1$. In ranks zero and one, every
family of flats has a common adapted basis.
\end{theorem}

The proof relies on the following two lemmas. 
We first treat finite families. Lemma~\ref{intersection-bases} builds
a finite set that is a basis on every intersection; Lemma~\ref{circuit-certificate}
shows that this set can be extended to a common adapted basis.

\begin{lemma}\label{intersection-bases}
Let $\cF=\{F_1,\ldots,F_m\}$ be a nonempty family of flats, and put
\begin{equation}\label{intersection-family}
 \cI=\left\{\bigcap_{i\in I}F_i:
          \varnothing\ne I\subseteq\{1,\ldots,m\}\right\}.
\end{equation}
Suppose that, for each $i$, there is a basis $B_i$ of $M|F_i$ adapted to
every $W\in\cI$ contained in $F_i$. Then there is a set
$Q\subseteq\bigcup_iF_i$ such that $Q\cap W$ is a basis of $W$ for
every $W\in\cI$.
\end{lemma}

\begin{proof}
For each \(W\in\mathcal I\), let
\[
T_W=\bigvee_{\substack{U\in\mathcal I\\U\subsetneq W}}U.
\]
Choose elements \(D_W\subseteq W\setminus T_W\) that extend a basis of \(T_W\) to a basis of \(W\). Thus
\[
|D_W|=\operatorname{rk}(W)-\operatorname{rk}(T_W),
\qquad
\operatorname{cl}(T_W\cup D_W)=W.
\]
Set \(Q=\bigcup_{W\in\mathcal I}D_W\). From the construction, it follows that the $D
_W$ are pairwise disjoint: take $W, U \in \mathcal{I}$ with $W \neq U$. Without loss of generality suppose \(U\nsubseteq W\), then
\[
U\cap W\subsetneq U,
\qquad U\cap W\subseteq T_U,
\]
so \(D_U\cap W=\varnothing\) and hence $D_U \cap D_W = \varnothing$. Consequently,
\[
Q\cap W=\bigsqcup_{\substack{U\in\mathcal I\\U\subseteq W}}D_U.
\]

We now show that \(Q\cap W\) is a basis of \(W\) by induction on inclusion. For minimal \(W\), this follows directly from the choice of \(D_W\). For the induction step, let \(U_1,\ldots,U_s\) be the maximal members of \(\mathcal I\) that are strictly contained in \(W\). Then
\[
T_W=U_1\vee\cdots\vee U_s,
\qquad
Q\cap W=Y\sqcup D_W,
\quad\text{where}\quad
Y=\bigcup_{j=1}^s(Q\cap U_j).
\]
By induction, \(Q\cap U_j\) is a basis of \(U_j\), so \(Y\) already spans \(T_W\). It remains to show $
|Y|=\operatorname{rk}(T_W)$.
Choose \(i\) such that \(W\subseteq F_i\). This is where we use the assumption of existence of the basis \(B_i\).
For every nonempty \(J\subseteq\{1,\ldots,s\}\), put
\[
V_J=\bigcap_{j\in J}U_j.
\]
This intersection belongs to \(\mathcal I\), is properly contained in \(W\), and lies in \(F_i\). Therefore by induction, \(Q\cap V_J\) is a basis of \(V_J\). Also by the assumption on \(B_i\), \(B_i\cap V_J\) is a basis of \(V_J\). Hence
\[
|Q\cap V_J|=|B_i\cap V_J|=\operatorname{rk}(V_J).
\]
Now we apply the inclusion–exclusion principle to the union of finite sets \(Q\cap U_j\):
\[
\begin{aligned}
|Y|
&=\sum_{\varnothing\ne J\subseteq[s]}
(-1)^{|J|+1}|Q\cap V_J|\\
&=\sum_{\varnothing\ne J\subseteq[s]}
(-1)^{|J|+1}\operatorname{rk}(V_J)\\
&=\sum_{\varnothing\ne J\subseteq[s]}
(-1)^{|J|+1}|B_i\cap V_J|\\
&=\left|\bigcup_{j=1}^s(B_i\cap U_j)\right|.
\end{aligned}
\]

Write \(A=\bigcup_j(B_i\cap U_j)\). Because \(A\subseteq B_i\), it is independent. Because \(B_i\cap U_j\) spans \(U_j\) for every \(j\), \(A\) spans \(T_W\). Thus
\[
|Y|=|A|=\operatorname{rk}(T_W),
\]
as desired. Finally, \(Y\cup D_W=Q\cap W\) spans \(W\), and
\[
|Q\cap W|
=\operatorname{rk}(T_W)
+\bigl(\operatorname{rk}(W)-\operatorname{rk}(T_W)\bigr)
=\operatorname{rk}(W).
\]
It is therefore a basis of \(W\).

The important point is that matroid rank does not generally satisfy inclusion–exclusion. The basis \(B_i\) allows the relevant ranks to be interpreted as cardinalities of subsets of one independent set, where inclusion–exclusion does apply.
\end{proof}

\begin{lemma}\label{circuit-certificate}
Let $M$ be a rank $r$ matroid, and let $\cF$ and $\cI$ be as in
\eqref{intersection-family}. Suppose $Q\subseteq\bigcup_iF_i$ and
$Q\cap W$ is a basis of $W$ for every $W\in\cI$. If $Q$ is dependent then
there is $J\subseteq\{1,\ldots,m\}$, with $|J|\le r+1$, such that
\begin{equation}\label{rank-certificate}
 \sum_{\varnothing\ne K\subseteq J}(-1)^{|K|+1}
      \rk\!\left(\bigcap_{k\in K}F_k\right)
 >\rk\!\left(\bigcup_{j\in J}F_j\right).
\end{equation}
In particular, the flats $F_j$, $j \in J$, have no common adapted basis.
\end{lemma}

\begin{proof}
Choose a circuit $C\subseteq Q$. Pick $F_j$, $j \in J$ such that $C \subset \bigcup_{j \in J} F_j$ and $|J|\le|C|\le r+1$. Put $U=\bigcup_{j\in J} F_j$. Note that $U$ may not be a flat itself. The finite set $Q\cap U$
spans $U$ and contains $C$, so $|Q\cap U|>\rk(U) := \rk(\cl(U))$.
The inclusion-exclusion applied to the sets $Q \cap F_j$, $j \in J$, shows that $|Q\cap U|$ coincides with the left-hand side of \eqref{rank-certificate} which proves the first claim. To prove the second claim, suppose a basis $B$ is adapted to the flats $F_j$, $j \in J$. Then, for any subset $K \subset J$, $B$ is adapted to $\bigcap_{k \in K} F_k$. Using
inclusion-exclusion for the subsets $B \cap F_j$, $j \in J$, we get $|B\cap U|=\rk(U)$ which contradicts \eqref{rank-certificate}.
\end{proof}

\begin{proof}[Proof of Theorem~\ref{adapted-helly}]
First suppose the family is finite. We prove the claim by induction on the rank $r$.
The cases $r=0,1$ are immediate.
Discard the bottom flat $0_M$ and the top flat $E$, to which every basis
is adapted. Write the remaining family as $\{F_1,\ldots,F_m\}$, and let $\cI$ be
as in \eqref{intersection-family}.

Fix $i$ and put $k=\rk(F_i)$, so $1\le k < r$. In the restriction $M|F_i$, consider
the flats $F_i\cap F_j$. Since $k+2 \leq r+1$, any $k+1$ of the $F_j$ together with $F_i$ have a common adapted basis. This implies that any $k+1$ of the $F_i \cap F_j$ have a common adapted basis. The induction hypothesis then gives a basis
$B_i$ of $F_i$ adapted to all the $F_i\cap F_j$, and hence to all their
intersections. In fact, if $W=\bigcap_{j\in J}F_j\subseteq F_i$,
then $W=\bigcap_{j\in J}(F_i\cap F_j)$, so $B_i$ is adapted to every
member of $\cI$ contained in $F_i$.

Now by Lemma~\ref{intersection-bases}, we have a set $Q\subseteq\bigcup_iF_i$
meeting each member of $\cI$ in a basis.
Lemma~\ref{circuit-certificate} and the hypothesis imply that $Q$ is independent.
Extend the set $Q$ to a basis $B$ of $M$. Since $Q\cap F_i$ already spans $F_i$,
the basis $B$ adds no element of $F_i$. Thus
$B\cap F_i=Q\cap F_i$, so $B$ is adapted to every flat $F_i$.

Now suppose $\cF$ is possibly infinite and satisfies the local hypothesis, that is, any $r+1$ subfamily has an adapted basis. By what we just proved, every
finite subfamily has a common adapted basis.
If $\cF$ contained $2^r+1$ distinct flats, these could not all admit
one adapted basis, since a basis, which has $r$ elements, is adapted to exactly $2^r$
flats. Hence $\cF$ itself has at most $2^r$ distinct members, and the
finite-family result applies.

To prove that the $r+1$ bound is sharp, take the uniform matroid $U_{r,r+1}$, whose independent
sets are the subsets of size at most $r$. Its $r+1$ singleton flats
have no common adapted basis, while any $r$ do. Note that $U_{r,r+1}$ is the graphic matroid of a cycle with $r+1$ edges.
\end{proof}

For any field $K$, take the rank $r$ matroid consisting of all vectors in $K^r$,
including the zero vector as a loop. Its flats are exactly the linear
subspaces. Theorem~\ref{adapted-helly} therefore recovers Klyachko's
theorem over every field.

A \emph{flag} is a chain of flats; a basis is adapted to a collection of
flags if it is adapted to every flat occurring in them. The following is an immediate corollary of Theorem \ref{adapted-helly}.

\begin{corollary}\label{flags-helly}
A collection of flags in a matroid of rank $r$ has a common adapted
basis if every subcollection of at most $r+1$ flags does.
\end{corollary}

%\begin{proof}
%Any at most $r+1$ flats occurring in the collection belong to at most
%$r+1$ flags. Apply Theorem~\ref{adapted-helly} to all the occurring flats. One-element
%flags from its sharpness example show that $r+1$ is sharp for $r\ge2$.
%\end{proof}

For a finite loopless matroid $M$ on $E$, the support of its
\emph{Bergman fan} is
\[
 \operatorname{Berg}(M)=\{w\in\mathbb R^E:
       \min_{e\in D}w_e\text{ is attained at least twice
       for every circuit }D\}.
\]
Its fine subdivision has cones indexed by flags of flats
\cite{ArdilaKlivans2006,FeichtnerSturmfels2005}.
For a basis $B$, its \emph{apartment} is
\[
 A_B=\{w\in\operatorname{Berg}(M):
          w(B)\ge w(B')\text{ for every basis }B'\},
 \qquad w(S)=\sum_{e\in S}w_e.
\]
A point $w$ lies in $A_B$ exactly when $B$ is adapted to every flat in
its flag of sets $\{e:w_e\ge t\}$
\cite[Definition~3.6 and Proposition~3.7]{KavehManon2024}.
Thus Corollary~\ref{flags-helly} gives the following.

\begin{corollary}\label{apartments}
A collection of points in the Bergman fan of a finite loopless
rank $r$ matroid lies in one apartment if every subcollection of at most
$r+1$ points does. For $r\ge2$, this is the optimal uniform bound over
rank $r$ matroids.
\end{corollary}

\begin{proof}
Apply Corollary~\ref{flags-helly} to the flags $\{e:w_e\ge t\}$, for $w$ in the given collection of points in the Bergman fan.
For sharpness, in $U_{r,r+1}$ take the singleton indicator vectors
$w_i=\mathbf1_{\{i\}}$. They belong to the Bergman fan, and
$w_i\in A_B$ exactly when $i\in B$. Any $r$ points therefore lie in an
apartment, whereas all $r+1$ do not.
\end{proof}

We next deduce a splitting result for tropical toric vector
bundles. This is a tropical extension of Klyachko's result on splitting of toric vector bundles over projective space \cite{Klyachko1990}. A tropical toric vector bundle is described by a map
$\Phi:|\Sigma|\to\operatorname{Berg}(M)$ that, on each cone of
$\Sigma$, is integral and linear in the coordinates of some apartment
\cite[Definition~4.1]{KavehManon2024}.
It is called \emph{split} when its entire image lies in one apartment
\cite[Definition~9.1]{KavehManon2024}. %In this case its extension class
%contains a direct sum of tropical line bundles
%\cite[Proposition~9.2]{KavehManon2024}.

\begin{corollary}\label{tropical-splitting}
A tropical toric vector bundle of rank $r<n$ on the fan of
$\mathbb P^n$ is split.
\end{corollary}

\begin{proof}
Any at most $r+1\le n$ rays belong to one cone. So their images lie in one apartment. Corollary~\ref{apartments} then implies that all the images of rays lie in one apartment $A_B$, for some basis $B$. It remains to show that the whole image of $\Phi$ lands in $A_B$. But this is a consequence of the piecewise linearity of $\Phi$.
%We must check that the whole image lies there.
%Fix a cone $\sigma$ and a local apartment basis $C$, and write
%$w(x)=\Phi(x)$. The coordinates $w_c$, %$c\in C$, are linear on
%$\sigma$; every other $w_e$ is a minimum of some of these linear
%coordinates \cite[Proposition~3.7]{KavehManon2024}. Thus every $w_e$
%is concave and positively homogeneous.
%Writing $x=\sum_\rho\lambda_\rho v_\rho$ with $\lambda_\rho\ge0$
%over its primitive ray generators $v_\rho$ gives
%\[
% w(x)(B)\ \ge\ \sum_\rho\lambda_\rho w(v_\rho)(B)
%   =\sum_\rho\lambda_\rho w(v_\rho)(C)=w(x)(C).
%\]
%The equality uses that both $B$ and $C$ maximize weight at each ray
%image. Since $C$ maximizes weight throughout $\sigma$, the reverse
%inequality holds. Thus $w(x)\in A_B$ for every $x\in\sigma$, as required.
\end{proof}

\begin{remark}
The condition $r<n$ is optimal for $n\ge2$. Let the underlying matroid
be $U_{n,n+1}$ on the ground set $\{0,\ldots,n\}$ and let $v_0,\ldots,v_n$ be the
rays of the fan of $\mathbb P^n$. On the maximal cone omitting $v_j$,
define, for $x=\sum_{i\ne j}\lambda_i v_i$ with $\lambda_i\ge0$,
\[
 \Phi_i(x)=\lambda_i\ (i\ne j),\qquad
 \Phi_j(x)=\min_{i\ne j}\lambda_i.
\]
These formulas agree on overlaps and are integral and linear in the
apartment coordinates of $A_{\{0,\ldots,n\}\setminus\{j\}}$.
They define a rank $n$ bundle whose images of rays are standard basis vectors. But the standard basis vectors do not lie in one apartment, which proves that it is not split.
\end{remark}

\section{Distributivity and common adapted bases}\label{sec:example}
A common adapted basis implies distributivity of the generated lattice
(Section~\ref{sec:basics}). The converse requires one additional
condition: rank must be a valuation. We prove this criterion in
Proposition~\ref{adapted-criterion} and deduce the full-height case
needed in Section~\ref{sec:distributive}.

\subsection{Adapted bases and valuations}

A real-valued function $v$ on a lattice $L$ is a \emph{valuation} if
\begin{equation}\label{valuation-identity}
 v(X)+v(Y)=v(X\wedge Y)+v(X\vee Y)\qquad\text{for all }X,Y\in L.
\end{equation}
This identity is automatic for comparable pairs. Cardinality is a
valuation on a lattice of finite sets under union and intersection,
and dimension is a valuation on the lattice of vector subspaces.
Matroid rank satisfies only
\[
 \rk(X)+\rk(Y)\ge\rk(X\wedge Y)+\rk(X\vee Y),
\]
both for sets, with join given by union, and for flats, with join
given by closure of the union. Thus rank is a valuation precisely
when equality holds for every pair in the lattice.

\subsection{Birkhoff's theorem}

For a finite poset $\Pi$, let $J(\Pi)$ be its lattice of order ideals:
subsets $I\subseteq\Pi$ such that $p\in I$ and $q\le p$ imply
$q\in I$. Its meet and join are intersection and union. For $p\in\Pi$,
write
\[
 I_p=\{q\in\Pi:q\le p\},\qquad I_p^-=I_p\setminus\{p\}.
\]
An element of a finite lattice is \emph{join-irreducible} if it is
not the least element and is not the join of two strictly smaller
elements. In $J(\Pi)$ these are exactly the principal ideals $I_p$.

\begin{theorem}[Birkhoff]\label{birkhoff}
Let $D$ be a finite distributive lattice, and let $\Pi$ be the set of
join-irreducible elements of $D$, ordered as in $D$. Then
\[
 D\longrightarrow J(\Pi),\qquad x\longmapsto\{p\in\Pi:p\le x\},
\]
is an isomorphism of lattices. Its inverse sends an ideal $I$ to the
join $\bigvee I$, the join of the empty ideal being the least element
of $D$.
\end{theorem}

See \cite[Chapter~5]{DaveyPriestley2002}. Every maximal chain in
$J(\Pi)$ has the form
\begin{equation}\label{maximal-chains}
 \varnothing=K_0\subsetneq K_1\subsetneq\cdots\subsetneq K_n=\Pi,
 \qquad K_k=K_{k-1}\cup\{p_k\},
\end{equation}
where $p_1,\ldots,p_n$ lists the elements of $\Pi$ in an order
compatible with the partial order. Indeed, between two ideals one
can add a minimal element of their difference. Hence $D$ has height
$|\Pi|$, and every ideal $I$ occurs at position $|I|$ on some maximal
chain.

\begin{example}\label{birkhoff-example}
Let $\Pi=\{a,b,c\}$ with $a<c$ and $b<c$, and $a,b$ incomparable.
Then
\[
 J(\Pi)=\{\varnothing,\{a\},\{b\},\{a,b\},\Pi\}.
\]
Its join-irreducibles are $\{a\},\{b\},\Pi$, and its two maximal
chains correspond to the orders $a,b,c$ and $b,a,c$. Both have
length $3$.
\end{example}

\begin{lemma}\label{increment-lemma}
Let $v$ be a valuation on $J(\Pi)$, and put
\[
 \delta_p(v)=v(I_p)-v(I_p^-)\qquad(p\in\Pi).
\]
\begin{enumerate}[label=\textup{(\alph*)},leftmargin=*,itemsep=1pt]
\item If $I$ and $I\cup\{p\}$ are ideals and $p\notin I$, then
 $v(I\cup\{p\})=v(I)+\delta_p(v)$.
\item For every ideal $I$,
 \[
  v(I)=v(\varnothing)+\sum_{p\in I}\delta_p(v).
 \]
\item Two valuations on $J(\Pi)$ that agree on one maximal chain agree
 everywhere.
\end{enumerate}
Conversely, for any real numbers $c$ and $c_p$ $(p\in\Pi)$, the
function $I\mapsto c+\sum_{p\in I}c_p$ is a valuation on $J(\Pi)$.
\end{lemma}

\begin{proof}
In (a), $I_p^-\subseteq I$ and $p\notin I$, so
\[
 I\cap I_p=I_p^-,\qquad I\cup I_p=I\cup\{p\}.
\]
Applying \eqref{valuation-identity} to $I,I_p$ gives (a). Adding the
elements of $I$ in an order compatible with the partial order gives
(b). Along a maximal chain \eqref{maximal-chains}, the successive
differences are the increments $\delta_{p_k}(v)$, each occurring
once. Thus the chain determines $v(\varnothing)$ and all increments,
which proves (c). The converse follows from inclusion--exclusion.
\end{proof}

\subsection{The criterion}

\begin{proposition}\label{adapted-criterion}
Let $M$ be a finite-rank matroid and $\cF$ a finite nonempty family of
its flats. Let $D=\langle\cF\rangle$, and let $\mathcal R$ be the
lattice of subsets generated by $\cF$ under set-theoretic union and
intersection. The following conditions are equivalent:
\begin{enumerate}[label=\textup{(\roman*)},leftmargin=*,itemsep=1pt]
\item $\cF$ has a common adapted basis.
\item $D$ is distributive and rank is a valuation on $D$.
\item Rank is a valuation on $\mathcal R$.
\end{enumerate}
\end{proposition}

The lattice $\mathcal R$ is automatically distributive, but its
members need not be flats: its join is union, rather than closure
of the union.

\begin{proof}
\textup{(i)$\Rightarrow$(ii).} Let $B$ be a common adapted basis.
By \eqref{coordinate-identities}, the flats $\cl(S)$, $S\subseteq B$,
form a Boolean sublattice containing $\cF$, and hence $D$. Under the
isomorphism $S\mapsto\cl(S)$, rank is cardinality. Thus $D$ is
distributive and rank is a valuation on $D$.

\textup{(i)$\Rightarrow$(iii).} For the same basis $B$, let
\[
 \mathcal S=\{A\subseteq E:A\subseteq\cl(B\cap A)\}.
\]
This family contains $\cF$ and is closed under unions. It is also
closed under intersections, since for $A,A'\in\mathcal S$,
\[
 A\cap A'\subseteq\cl(B\cap A)\cap\cl(B\cap A')
 =\cl(B\cap A\cap A')
\]
by \eqref{coordinate-identities}. Hence $\mathcal R\subseteq\mathcal S$,
so $B\cap A$ is a basis of $A$ for every $A\in\mathcal R$. Therefore
$\rk(A)=|B\cap A|$ on $\mathcal R$, a valuation by
inclusion--exclusion.

\textup{(ii)$\Rightarrow$(i).} Since $D$ is finitely generated and
distributive, it is finite. Write $D=J(\Pi)$ by
Theorem~\ref{birkhoff}, and denote the flat corresponding to $I$ by
$F_I$. Thus
\[
 F_{I\cup I'}=F_I\vee F_{I'},\qquad
 F_{I\cap I'}=F_I\cap F_{I'}.
\]
The least flat $F_\varnothing$ of $D$ need not equal $0_M$. Choose a
basis $B_\varnothing$ of it. For each $p\in\Pi$, extend a basis of
$F_{I_p^-}$ to a basis of $F_{I_p}$, and let $A_p$ be the set of
added elements. Then
\[
 |A_p|=\delta_p:=\rk(F_{I_p})-\rk(F_{I_p^-}),
\]
and
\begin{equation}\label{new-elements}
 \cl(F_{I_p^-}\cup A_p)=F_{I_p}.
\end{equation}
For each ideal $I$, put
\[
 B_I=B_\varnothing\cup\bigcup_{p\in I}A_p.
\]
We show by induction on $|I|$ that $B_I$ is a basis of $F_I$.
The case $I=\varnothing$ holds by construction. Otherwise choose a
maximal element $p$ of $I$ and put $I'=I\setminus\{p\}$. By induction,
$B_{I'}$ spans $F_{I'}$, which contains $F_{I_p^-}$. Thus
\eqref{new-elements} shows that $B_I=B_{I'}\cup A_p$ spans
\[
 F_{I'}\vee F_{I_p}=F_I.
\]
Also $B_I\subseteq F_I$, and Lemma~\ref{increment-lemma}(a) gives
\[
 |B_I|\le|B_{I'}|+|A_p|
 =\rk(F_{I'})+\delta_p=\rk(F_I).
\]
Hence $B_I$ is a basis of $F_I$. Extend $B_\Pi$ to a basis $B$ of
$M$. For every $I$, the independent set $B\cap F_I$ contains the
basis $B_I$ of $F_I$, so $B\cap F_I=B_I$. Thus $B$ is adapted to
all members of $D$.

\textup{(iii)$\Rightarrow$(i).} The lattice $\mathcal R$ is finite,
since every member is a union of intersections of the finitely many
generators. Adjoin $\varnothing,E$ to obtain $\mathcal R^+$. Rank
remains a valuation, since the added sets are comparable with every
member. Choose a maximal chain
\[
 \varnothing=C_0\subsetneq C_1\subsetneq\cdots\subsetneq C_n=E.
\]
Starting with $B_0=\varnothing$, successively extend a basis $B_{k-1}$
of $C_{k-1}$ to a basis $B_k$ of $C_k$, and put $B=B_n$.
Then $B\cap C_k=B_k$, so
\[
 |B\cap C_k|=\rk(C_k)\qquad(0\le k\le n).
\]
The valuations $A\mapsto\rk(A)$ and $A\mapsto|B\cap A|$ on the finite
distributive lattice $\mathcal R^+$ agree on this chain, and hence
everywhere by Theorem~\ref{birkhoff} and
Lemma~\ref{increment-lemma}(c). In particular $B$ is adapted to every
$F\in\cF$.
\end{proof}

\begin{example}\label{small-example}
Let $M=U_{3,4}$ on $E=\{e_1,e_2,e_3,e_4\}$, and take
$F=\{e_1,e_2\}$ and $G=\{e_2,e_3\}$. Then
\[
 \langle F,G\rangle=\bigl\{\{e_2\},F,G,E\bigr\}
\]
is Boolean, and rank is a valuation since $2+2=1+3$. In the preceding
construction the two join-irreducibles have increments one. Taking
$B_\varnothing=\{e_2\}$, $A_F=\{e_1\}$ and $A_G=\{e_3\}$ gives the
common adapted basis $\{e_1,e_2,e_3\}$.
\end{example}

\begin{corollary}\label{full-height}
A finite distributive sublattice of height $r$ in the lattice of flats
of a rank $r$ matroid has a common adapted basis.
\end{corollary}

\begin{proof}
Write the sublattice as $J(\Pi)$, with corresponding flats $F_I$.
Its height is $|\Pi|=r$. Along any maximal chain there are $r$ strict
rank increases, whose sum is at most $r$. Each increase is therefore
one, and the least flat has rank zero. Since every ideal $I$ occurs
at position $|I|$ on some maximal chain, $\rk(F_I)=|I|$ for all $I$.
Thus rank is a valuation, and Proposition~\ref{adapted-criterion}
applies.
\end{proof}

\subsection{The rank condition cannot be omitted}

Dimension is a valuation on the lattice of vector subspaces, so in
that setting distributivity alone gives a common adapted basis
\cite[\S2.2.2]{Klyachko1990}. For matroid flats the additional rank
condition is essential.

\begin{example}\label{four-cycle-example}
In $U_{3,4}$, take $F=\{e_1,e_2\}$ and $G=\{e_3,e_4\}$. These flats
generate the Boolean lattice $\{\varnothing,F,G,E\}$, but
\[
 \rk(F)+\rk(G)=4>3=\rk(F\wedge G)+\rk(F\vee G).
\]
A common adapted basis would contain all four elements, which is
impossible in rank three.
\end{example}

Rank three is minimal for this phenomenon. In rank at most two, any
incomparable flats have rank one, their meet has rank zero and their
join has rank two. Thus rank is a valuation on the whole lattice of
flats, and Proposition~\ref{adapted-criterion} applies to every finite
family generating a distributive lattice.

The example is representable over every field $K$: for a basis
$v_1,v_2,v_3$ of $K^3$, use the vectors
\[
 v_1,\quad v_2,\quad v_3,\quad v_1+v_2+v_3.
\]
Any three are independent. The spans of $F$ and $G$ meet in the line
$\langle v_1+v_2\rangle$, although $F\cap G=\varnothing$. Thus the
span map preserves joins but not meets. The ambient basis
$v_1,v_1+v_2,v_3$ is adapted to both planes, but it uses a vector
outside the representing set.

\section{Helly numbers for distributive generation}\label{sec:distributive}

In this section we prove quadratic bounds for $h_{\mathrm{dist}}(r)$.
Recall that a family $\cF$ is distributive if $\langle\cF\rangle$
satisfies
\begin{equation}\label{distributive-law}
 x\wedge(y\vee z)=(x\wedge y)\vee(x\wedge z)
\end{equation}
for all $x,y,z$, or equivalently the dual identity. The upper bound
uses Theorem~\ref{adapted-helly} and a bound on irredundant generators
of finite distributive lattices. The lower bound comes from graphic
matroids.

\begin{theorem}\label{dist-bounds}
For every $r\ge2$,
\[
 h_{\mathrm{dist}}(r)\ge
 \begin{cases}
 r+1,&2\le r\le6,\\
 \left\lfloor(r-1)^2/4\right\rfloor,&r\ge7,
 \end{cases}
\qquad
 h_{\mathrm{dist}}(r)\le
 \begin{cases}
 r+1,&2\le r\le4,\\
 2r-3,&5\le r\le7,\\
 \left\lfloor(r-1)^2/4\right\rfloor+1,&r\ge8.
 \end{cases}
\]
The same bounds hold when the distributive Helly number is restricted to
graphic matroids, or to matroids representable over any prescribed field.
\end{theorem}

Thus $h_{\mathrm{dist}}(r)=(r-1)^2/4+O(1)$, whereas
$h_{\mathrm{cb}}(r)=r+1$. The distributive bounds coincide for
$2\le r\le4$ and differ by one for $r\ge8$.
Table~\ref{tab:numbers} gives the values in small ranks.

\begin{table}[ht]
\centering
\small
\begin{tabular}{@{}lrrrrrrrrr@{}}
\toprule
Rank $r$&2&3&4&5&6&7&8&9&10\\
\midrule
$h_{\mathrm{cb}}(r)$&3&4&5&6&7&8&9&10&11\\
Lower bound for $h_{\mathrm{dist}}(r)$&3&4&5&6&7&9&12&16&20\\
Upper bound for $h_{\mathrm{dist}}(r)$&3&4&5&7&9&11&13&17&21\\
\bottomrule
\end{tabular}
\caption{The Helly numbers and bounds proved in this paper.}
\label{tab:numbers}
\end{table}

\subsection{Minimal failures}\label{subsec:minimal-failures}

A \emph{minimal failure of distributive generation} is a finite
family that is not distributive but whose proper subfamilies are.
A family $\mathcal X$ in a lattice is \emph{irredundant} if no member
belongs to the sublattice generated by the others. As before, we use
only binary meets and joins, without adjoining endpoints.

\begin{lemma}\label{minimal-failures}
\begin{enumerate}[label=\textup{(\alph*)},leftmargin=*,itemsep=1pt]
\item Every family of flats that is not distributive contains a
 minimal failure.
\item A minimal failure and all its subfamilies are irredundant.
\item If a rank $r$ matroid has a minimal failure with $m$ members,
 then $h_{\mathrm{dist}}(r)\ge m$. If every minimal failure in every
 rank $r$ matroid has at most $N$ members, then
 $h_{\mathrm{dist}}(r)\le N$.
\end{enumerate}
\end{lemma}

\begin{proof}
(a) A violation of \eqref{distributive-law} involves three lattice
terms, hence only finitely many generators. Choose a minimal
subfamily among these generators that is not distributive.

(b) If a member were generated by the others, deleting it would leave
the same nondistributive lattice, contrary to minimality.
Every subfamily of an irredundant family is irredundant.

(c) A minimal failure with $m$ members disproves the Helly property
for every $h<m$. Conversely, by (a), every failure contains a minimal
one; if its size is at most $N$, it violates the local hypothesis for
$h=N$.
\end{proof}

\begin{example}\label{singleton-failure}
The $r+1$ singleton flats of $U_{r,r+1}$, $r\ge2$, form a minimal
failure. Every proper subfamily has a common adapted basis and is
therefore distributive. For the whole family, take distinct
singletons $F,G$ and let $Z$ be the join of the remaining $r-1$.
Then $Z\vee G=E$, while $F\wedge Z=F\wedge G=\varnothing$, so
\[
 F\wedge(Z\vee G)=F\ne\varnothing=(F\wedge Z)\vee(F\wedge G).
\]
Since $U_{r,r+1}$ is the cycle matroid of a cycle of length $r+1$,
this gives $h_{\mathrm{dist}}(r)\ge r+1$ already for graphic matroids.
\end{example}

\subsection{Irredundant generators}\label{subsec:generators}

Let $g(n)$ be the supremum of the sizes of irredundant families in
finite distributive lattices of height at most $n$, and put
\[
 f(n)=\max\{2n-2,\lfloor n^2/4\rfloor\}\qquad(n\ge1).
\]
We will prove that $g(n)=f(n)$ for $n\ge3$. By Birkhoff's theorem,
we may first consider lattices of sets. For $\alpha,\beta$ in a finite
set $\Omega$, the subsets satisfying
\begin{equation}\label{implication-sublattice}
 \alpha\in S\ \Longrightarrow\ \beta\in S
\end{equation}
form a sublattice of $2^\Omega$. Thus a set that violates an
implication satisfied by all the others cannot be generated by them.

A directed graph is \emph{reduced} if it has no loops or parallel
arcs, and no arc whose head can be reached from its tail by a directed
path avoiding that arc.

\begin{lemma}\label{pair-witnesses}
Let $S_1,\ldots,S_m$ be an irredundant family of nonempty proper
subsets of a finite set $\Omega$. For each $i$ there are
$\alpha_i\in S_i$ and $\beta_i\notin S_i$ such that every $S_j$ with
$j\ne i$ satisfies $\alpha_i\in S_j\Rightarrow\beta_i\in S_j$.
The arcs $\alpha_i\to\beta_i$ are distinct and form a reduced
directed graph on $\Omega$.
\end{lemma}

\begin{proof}
Suppose no such pair exists for $S_i$. For every $\alpha\in S_i$ and
$\beta\notin S_i$, some $S_j$, $j\ne i$, contains $\alpha$ but not
$\beta$. For each $\alpha\in S_i$, let $N_\alpha$ be the intersection
of the sets $S_j$, $j\ne i$, containing $\alpha$. This collection of sets is nonempty since $S_i$ is proper, and
$\alpha\in N_\alpha\subseteq S_i$. Hence
\[
 S_i=\bigcup_{\alpha\in S_i}N_\alpha,
\]
contrary to irredundance.

Choose such a \emph{witness pair} $(\alpha_i,\beta_i)$ for each $i$.
Its arc is not a loop. Two arcs cannot agree, since the corresponding
sets would both violate the same implication. Also, $S_i$ satisfies
every implication except its own. A directed path from $\alpha_i$
to $\beta_i$ using other arcs would therefore imply
$\beta_i\in S_i$, a contradiction.
\end{proof}

\begin{example}\label{path-example}
On $\Omega=\{1,2,3\}$ the sets
\[
 \{1\},\qquad\{1,2\},\qquad\{3\},\qquad\{2,3\}
\]
have witness pairs $(1,2)$, $(2,3)$, $(3,2)$ and $(2,1)$,
respectively. Thus they are irredundant. Their witness graph is the
path $1,2,3$ with both orientations of each edge; it has $2\cdot3-2$
arcs.
\end{example}

\begin{lemma}\label{arc-bound}
A reduced directed graph on $t\ge1$ vertices has at most $f(t)$ arcs.
\end{lemma}

\begin{proof}
Let the strongly connected components have sizes $s_1,\ldots,s_k$,
with $\sum_c s_c=t$. In a component of size $s$, choose a root and
spanning trees directed towards and away from it. Their union is
strongly connected and has at most $2s-2$ arcs. Every arc of the
component must belong to this union, since otherwise it would have
an alternative directed path. Thus there are at most $2s-2$ internal
arcs.

There is at most one arc from one component to another: a second
such arc, together with paths inside the two components, would make
the first redundant. The condensation is acyclic. If its underlying
graph contained a triangle, its acyclic orientation would give an
arc and an alternative two-step route, again contradicting
reducedness. Mantel's theorem \cite[Chapter~7]{Diestel2017} therefore
bounds the number of arcs between components by $\lfloor k^2/4\rfloor$.
The total is at most
\[
 \sum_{c=1}^k(2s_c-2)+\left\lfloor\frac{k^2}{4}\right\rfloor
 =\left\lfloor 2(t-k)+\frac{k^2}{4}\right\rfloor.
\]
The expression inside the floor is convex in $k$. Its endpoint
values on $1\le k\le t$ are $2t-2+1/4$ and $t^2/4$, which give
the bound $f(t)$.
\end{proof}

\begin{proposition}\label{generator-bound}
We have $g(1)=2$, $g(2)=3$, and, for $n\ge3$,
\[
 g(n)\le f(n)=\max\left\{2n-2,\left\lfloor\frac{n^2}{4}\right\rfloor\right\}.
\]
\end{proposition}

\begin{proof}
A lattice of height at most one has at most two elements, and the
two-element chain is irredundant. A distributive lattice of height
two is either the three-element chain or the four-element Boolean
lattice: any three elements strictly between its endpoints would
form an $M_3$. The chain is irredundant, whereas the four elements of
the Boolean lattice are not. Hence $g(1)=2$ and $g(2)=3$.

Let $\mathcal X$ be a nonempty irredundant family, and replace its
ambient lattice by $D=\langle\mathcal X\rangle$, of height $t\le n$.
If $\mathcal X$ contains neither endpoint of $D$, Birkhoff's theorem
(Theorem~\ref{birkhoff}) identifies its members with nonempty proper
ideals of a poset of size $t$. They remain irredundant as subsets of
that poset, since meets and joins are intersections and unions.
Lemmas~\ref{pair-witnesses} and~\ref{arc-bound} give
$|\mathcal X|\le f(t)\le f(n)$.

If $\mathcal X$ contains the greatest element $T$ of $D$ and has
another member, then $\bigvee(\mathcal X\setminus\{T\})<T$ by
irredundance. The lattice generated by the remaining members has
height at most $t-1$, since each of its chains extends by $T$.
Thus $|\mathcal X|\le1+g(n-1)$. The singleton case is immediate,
and the least-element case is dual. Consequently
\[
 g(n)\le\max\{f(n),1+g(n-1)\}\qquad(n\ge3).
\]
For $n=3$ this gives $g(3)\le4=f(3)$. Since
$f(n)\ge f(n-1)+1$ for $n\ge4$, induction completes the proof.
\end{proof}

\begin{remark}\label{generator-sharp}
The bound is sharp for every $n\ge3$. Choose a tree on an $n$-element
set $\Omega$. For each oriented edge $\alpha\to\beta$, take the
component containing $\alpha$ after deleting that edge. Among the
$2n-2$ resulting sets, only this one violates
$\alpha\in S\Rightarrow\beta\in S$, so they are irredundant by
\eqref{implication-sublattice}. Example~\ref{path-example} is the
three-vertex case.

For $n\ge4$, partition $\Omega=\Xi\sqcup\Upsilon$ as equally as
possible, with $\Xi=\{\xi_1,\ldots,\xi_p\}$ and
$\Upsilon=\{\eta_1,\ldots,\eta_q\}$. The $pq$ sets
\[
 S_{ij}=\{\xi_i\}\cup(\Upsilon\setminus\{\eta_j\})
\]
are irredundant, since only $S_{ij}$ violates
$\xi_i\in S\Rightarrow\eta_j\in S$. As $pq=\lfloor n^2/4\rfloor$,
we obtain $g(n)=f(n)$ for all $n\ge3$.
\end{remark}

\subsection{The upper bound}\label{subsec:upper}

\begin{proposition}\label{dist-upper}
For every $r\ge2$,
\[
 h_{\mathrm{dist}}(r)\le
 \max\left\{r+1,\,2r-3,\,
           \left\lfloor\frac{(r-1)^2}{4}\right\rfloor+1\right\}.
\]
\end{proposition}

\begin{proof}
Let $\cF$ be a minimal failure with $m$ members in a rank $r$ matroid.
By Lemma~\ref{minimal-failures}(c), it suffices to bound $m$; assume
$m>r+1$. Some one-member deletion $\cF'$ has no common adapted basis.
Otherwise every subfamily of at most $r+1$ members would have one,
and Theorem~\ref{adapted-helly} would give a common adapted basis
for $\cF$, contradicting its nondistributivity.

The lattice $D'=\langle\cF'\rangle$ is finite and distributive, and
$\cF'$ is irredundant. Its height is at most $r-1$: height $r$ would
give a common adapted basis by Corollary~\ref{full-height}. Hence
$m-1\le g(r-1)$. For $r=2,3$, Proposition~\ref{generator-bound}
gives $m\le r+1$, a contradiction. For $r\ge4$ it gives
\[
 m\le1+f(r-1)
 =\max\left\{2r-3,\,\left\lfloor\frac{(r-1)^2}{4}\right\rfloor+1\right\},
\]
as required.
\end{proof}

\subsection{The graphic lower bound}\label{subsec:lower}

Let $p,q\ge2$. The graph $G_{p,q}$ consists of $q+1$ internally
disjoint paths from $u$ to $v$. The $q$ short paths are $u,x_j,v$,
with edges $b_j=ux_j$ and $c_j=x_jv$. The long path is
$u,y_1,\ldots,y_p,v$, with edges $a_1=uy_1$,
$a_i=y_{i-1}y_i$ for $2\le i\le p$, and $w=y_pv$.
Its cycle matroid $M_{p,q}$ has rank $p+q+1$.
Figure~\ref{fig:graphic-obstruction} shows $G_{3,3}$.

\begin{figure}[ht]
\centering
\begin{tikzpicture}[
  x=1cm,y=1cm,
  every node/.style={font=\scriptsize},
  vertex/.style={circle,fill,inner sep=1.5pt}
]
\node[vertex,label=left:$u$] (u) at (0,0) {};
\node[vertex,label=right:$v$] (v) at (6,0) {};
\node[vertex,label=above:$x_1$] (x1) at (3,1.25) {};
\node[vertex,label=above:$x_2$] (x2) at (3,.42) {};
\node[vertex,label=below:$x_3$] (x3) at (3,-.42) {};
\node[vertex,label=below:$y_1$] (y1) at (1.25,-1.5) {};
\node[vertex,label=below:$y_2$] (y2) at (3,-1.5) {};
\node[vertex,label=below:$y_3$] (y3) at (4.75,-1.5) {};
\draw (u)--node[above left] {$b_1$} (x1)
         --node[above right] {$c_1$} (v);
\draw (u)--node[above] {$b_2$} (x2)
         --node[above] {$c_2$} (v);
\draw (u)--node[below] {$b_3$} (x3)
         --node[below] {$c_3$} (v);
\draw (u)--node[left] {$a_1$} (y1)
         --node[above] {$a_2$} (y2)
         --node[above] {$a_3$} (y3)
         --node[right] {$w$} (v);
\end{tikzpicture}
\caption{The graph $G_{3,3}$ in the lower-bound construction.}
\label{fig:graphic-obstruction}
\end{figure}
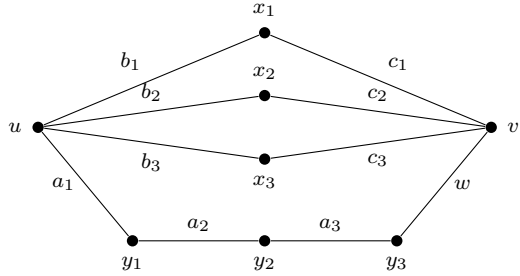

Over any field, $M_{p,q}$ is represented by
\begin{equation}\label{graphic-representation}
 a_i,\qquad b_j,\qquad c_j=b_j+z,\qquad
 w=a_1+\cdots+a_p+z,
\end{equation}
where $a_1,\ldots,a_p,b_1,\ldots,b_q,z$ are independent auxiliary
vectors and $z$ is not a ground element. Indeed, assign potentials
$0,z,-b_j$ and $-(a_1+\cdots+a_i)$ to $u,v,x_j$ and $y_i$,
respectively. Orient $b_j,a_i$ towards $u$ along their paths and
$c_j,w$ towards $v$. The edge differences are the displayed vectors,
and the nonzero vertex potentials form a basis, giving a change of
coordinates in the reduced incidence representation.

Set $\Omega=\Xi\sqcup\Upsilon$, where
$\Xi=\{\xi_1,\ldots,\xi_p\}$ and
$\Upsilon=\{\eta_1,\ldots,\eta_q\}$. For $S\subseteq\Omega$, put
\[
 R_S=\{a_i:\xi_i\in S\}\cup\bigcup_{\eta_j\in S}\{b_j,c_j\},
 \qquad F_S=\cl(R_S).
\]

\begin{lemma}\label{closure-lemma}
For every $S\subseteq\Omega$,
\begin{equation}\label{closure-formula}
 F_S=\begin{cases}
 R_S\cup\{w\},&\Xi\subseteq S\text{ and }S\cap\Upsilon\ne\varnothing,\\
 R_S,&\text{otherwise},
 \end{cases}
 \qquad
 \rk(F_S)=|S|+\mathbf1_{S\cap\Upsilon\ne\varnothing}.
\end{equation}
The map $S\mapsto F_S$ is injective and satisfies
$F_S\vee F_T=F_{S\cup T}$.
\end{lemma}

\begin{proof}
By \eqref{graphic-representation}, $R_S$ spans the independent
auxiliary vectors $a_i$ with $\xi_i\in S$, $b_j$ with
$\eta_j\in S$, and also $z$ when $S\cap\Upsilon\ne\varnothing$.
This gives the rank formula. A ground element belongs to this span
exactly when all auxiliary vectors in its expression occur among
these vectors, giving the closure formula. Finally,
$F_S\setminus\{w\}=R_S$ determines $S$, and
\[
 F_S\vee F_T=\cl(R_S\cup R_T)=F_{S\cup T}.
\]
\end{proof}

\begin{proposition}\label{graphic-obstruction}
The $pq$ distinct flats
\[
 H_{ij}=F_{\{\xi_i\}\cup(\Upsilon\setminus\{\eta_j\})}
 =\{a_i\}\cup\bigcup_{k\ne j}\{b_k,c_k\}
 \qquad(1\le i\le p,\ 1\le j\le q)
\]
form a minimal failure of distributive generation in $M_{p,q}$.
\end{proposition}

\begin{proof}
The displayed description and distinctness follow from
Lemma~\ref{closure-lemma}, since $p\ge2$.

First omit one member $H_{ij}$. The labels of all remaining members
satisfy $\xi_i\in S\Rightarrow\eta_j\in S$ and hence belong to the
distributive lattice $L_{ij}$ of subsets satisfying this implication.
The map $S\mapsto F_S$ preserves joins by Lemma~\ref{closure-lemma}.
It also preserves meets on $L_{ij}$: we have
$R_S\cap R_T=R_{S\cap T}$, and if $w\in F_S\cap F_T$, both labels
contain $\Xi$, so the implication puts $\eta_j$ in both. Thus
$w\in F_{S\cap T}$. Together with injectivity, this shows that the
image of $L_{ij}$ is a distributive sublattice containing all
remaining $H_{kl}$. Every proper subfamily is therefore distributive.

The full family generates
\[
 X_i=\bigcap_{j=1}^q H_{ij}=\{a_i\},\qquad
 Y_j=\bigcap_{i=1}^p\ \bigcap_{k\ne j}H_{ik}=\{b_j,c_j\}.
\]
Here $q\ge2$ ensures the second intersection has a nonempty index
set, and $p\ge2$ removes all $a_i$. Put
$F_\Xi=\bigvee_iX_i=\{a_1,\ldots,a_p\}$. For distinct $j,k$,
Lemma~\ref{closure-lemma} gives
\[
 F_\Xi\vee(Y_j\wedge Y_k)=F_\Xi
 \ \subsetneq\ (F_\Xi\vee Y_j)\wedge(F_\Xi\vee Y_k)
 =F_\Xi\cup\{w\}.
\]
Thus the dual distributive law fails.
\end{proof}

\begin{proof}[Completion of the proof of Theorem~\ref{dist-bounds}]
Example~\ref{singleton-failure} gives the lower bound $r+1$. For
$r\ge5$, choose $p=\lfloor(r-1)/2\rfloor$ and
$q=\lceil(r-1)/2\rceil$. Then $p,q\ge2$, and
Proposition~\ref{graphic-obstruction} gives a minimal failure of size
$pq=\lfloor(r-1)^2/4\rfloor$ in rank $r$. This bound exceeds $r+1$
exactly when $r\ge7$. The upper bounds are the values of the maximum
in Proposition~\ref{dist-upper}.

Restricting to a class of matroids preserves the upper bounds. Both
lower-bound constructions are graphic, and graphic matroids are
representable over every field \cite[Chapter~5]{Oxley2011}, so the
restricted Helly numbers satisfy the same bounds.
\end{proof}

\end{document}